\documentclass[11pt]{article}
\usepackage[T1]{fontenc}
\usepackage{lmodern}
\usepackage{amsmath,amssymb,amsthm}
\usepackage[margin=1.1in]{geometry}
\usepackage{microtype}
\usepackage{booktabs}
\usepackage{needspace}
\usepackage[authoryear,round]{natbib}
\usepackage[colorlinks=true,citecolor=blue,linkcolor=blue,urlcolor=blue]{hyperref}
\hypersetup{pdftitle={A remoteness formula for generalized Wythoff games},
  pdfauthor={Raul Sanchez Galan and Yijie Xu}}
\newtheorem{theorem}{Theorem}[section]
\newtheorem{proposition}[theorem]{Proposition}
\newtheorem{lemma}[theorem]{Lemma}
\newtheorem{corollary}[theorem]{Corollary}
\theoremstyle{definition}
\newtheorem{definition}[theorem]{Definition}
\newtheorem{remark}[theorem]{Remark}
\newtheorem{example}[theorem]{Example}
\newcommand{\R}{\mathcal R}
\newcommand{\Nzero}{\mathbb Z_{\geq 0}}
\newcommand{\WYT}{\mathrm{WYT}}
\newcommand{\mexb}{\operatorname{mex}_b}
\newcommand{\idx}{\iota_{a,b}}
\newcommand{\eps}{\varepsilon}
\DeclareMathOperator{\Var}{Var}

\begin{document}
\title{A remoteness formula for generalized Wythoff games}
\author{Ra\'ul S\'anchez Gal\'an%
\thanks{4i Intelligent Insights, Tecnoincubadora Marie Curie,
PCT Cartuja, 41092 Sevilla, Spain. }
\and
Yijie Xu%
\thanks{The Hong Kong University of Science and Technology (Guangzhou),
Guangzhou, China.}}

\maketitle

\begin{abstract}
Smith's remoteness function records how long an impartial game lasts when
the winner wins as quickly as possible and the loser delays defeat. For the
generalized Wythoff game $\WYT(a,b)$, in which two positive integer
parameters $a$ and $b$ determine which simultaneous removals from the two
piles are allowed, we give a formula for the remoteness of every position,
building on the known values at losing positions. The formula depends on
two indices: a local index determined by the coordinates of the losing
positions in a window of length $b$ ending at the smaller pile, and the
integer part of the pile difference divided by $a$. When $b=1$,
complementary Beatty sequences make the local index explicit, which gives
an explicit formula for $\WYT(a,1)$, including classical Wythoff's game.
We also show that if each attempted move in a finite impartial game
independently fails with a fixed probability and passes the turn, the
probability that the player to move wins under optimal play is an explicit
exponential function of remoteness. For every success probability strictly
between $0$ and $1$, the optimal attempts are exactly the moves that
reduce remoteness by one.
\end{abstract}



\section{Introduction}

In Wythoff's game, a player removes a positive number of tokens from one
of two piles, or the same positive number from both. The player with no
legal move loses. \citet{Wythoff1907} showed that the losing
positions, with the smaller pile first, are
\[
 (\lfloor m\phi\rfloor,\lfloor m\phi^2\rfloor),\qquad
 m\in\Nzero,\qquad \phi=\frac{1+\sqrt5}{2}.
\]
This determines who can win, but it does not distinguish between winning
moves that end the game at different times.

Smith's remoteness function $\R$ \citep{Smith1966} makes this distinction.
At a winning position its value is the number of moves needed to force
a win against an opponent who delays defeat; at a losing position it is
the number of moves for which defeat can be delayed against an opponent
who wins as quickly as possible. Its parity gives the usual classification
into P-positions, which are losing for the player to move, and
N-positions, which are winning. A move that decreases the remoteness by
exactly one is called \emph{remoteness-optimal}: from an N-position it
forces the fastest win, and from a P-position it makes the game last as
long as possible. Smith credited the underlying idea to
\citet{Steinhaus1925} and called the function the Steinhaus remoteness function. It governs conjunctive compounds, in which a player moves in every component and play stops when some component ends, in the way that
the Sprague--Grundy function governs disjunctive sums
\citep{Smith1966,Conway1976,WinningWays}.

Gurvich's two-parameter family $\WYT(a,b)$ \citep{Gurvich2012}, with
positive integers $a$ and $b$, generalizes Wythoff's game as follows. A
move removes $e$ tokens from the first pile and $f$ from the second, not
both zero, and is allowed when $|e-f|<a$ or $\min(e,f)<b$. In words, a
player may remove tokens from one pile, or from both piles provided that
the two amounts differ by less than $a$ or the smaller amount is less
than $b$. For $a=b=1$ this is Wythoff's game.

\citet{BGMV2024} studied remoteness for several impartial games; related
recent work treats variants of $k$-NIM \citep{GMMV2023,GMMN2024}.    For $\WYT(a,b)$, let $(x_0,y_0)=(0,0),(x_1,y_1),\dots$ be the
P-positions with $x\leq y$, ordered by increasing smaller pile. \citet{BGMV2024} proved that
$(x_m,y_m)$ has remoteness $2m$. In the same paper they gave a
polynomial algorithm for N-positions: it identifies at most six
P-positions reachable in one move and takes the one of smallest
remoteness. They state explicitly that they have no closed formula for
the remoteness of N-positions.

The main result here reduces that selection to the minimum of two
indices: a local index associated with the smaller pile, and the
integer part of the pile difference divided by $a$
(Theorem~\ref{thm:main}). The local index is defined by a window of
length $b$ in the P-position coordinates. A direct consequence of the
minimum-excludant recursion shows that this window always contains a
coordinate. The same observation proves that one of the two candidate
indices is reachable and that no smaller index can be reached.

The contribution concerns the remoteness of N-positions; the even
values are the result of \citet{BGMV2024}. For $b>1$, our formula still
uses the P-position sequences and does not supply an elementary closed
form for those sequences. For $b=1$, Fraenkel's formulas
\citep{Fraenkel1982} make both indices explicit, including the
classical Wythoff case.

We also give a probabilistic interpretation of remoteness. Suppose each
attempted move succeeds independently with probability $p$, and a
failed attempt leaves the position unchanged and passes the turn. Under
optimal play, the probability that the player to move wins is an
exponential function of the deterministic remoteness
(Theorem~\ref{thm:random}); this holds for every finite impartial
normal-play game. In particular, this probability exceeds $1/2$ exactly
at the N-positions, and for every $p\in(0,1)$ the optimal attempts are
exactly the remoteness-optimal moves. Applied to $\WYT(a,b)$,
Theorems~\ref{thm:main} and~\ref{thm:random} together give the optimal
winning probability in terms of P-position coordinates
(equation~\eqref{eq:random-wyt}).

Several other randomizations of combinatorial games use different
rules. \citet{CP2020} study transmission errors that can replace an
intended move by another legal move. \citet{DT2025}
study Chomp and Nim when both players choose legal moves uniformly
at random. In the random-turn games of \citet{PSSW2007}, a coin toss
decides which player moves next. Here the players choose their
attempted moves optimally; the only randomness is whether an attempt
succeeds, and a failure passes the turn without changing the position.
Our value formula depends on this specific failure rule and on a
common success probability for every attempted move.

\section{Remoteness and generalized Wythoff games}\label{sec:prelim}

A finite impartial normal-play game is represented by a finite acyclic
directed graph. Its vertices are positions and its edges are legal
moves. We write $s\to t$ when $t$ is an option of $s$. A terminal
position has no options, and the player to move there loses. A
P-position is losing for the player to move; an N-position is winning.
Thus every option of a P-position is an N-position, and every
N-position has a P-position option. The arguments also apply when the
whole graph is infinite but the subgraph reachable from each initial
position is finite.

\begin{definition}\label{def:remote}
The \emph{remoteness function} is the function $\R$ on positions given
by $\R(s)=0$ at terminal positions and, at every other position, by
\begin{equation}\label{eq:remote}
 \R(s)=
 \begin{cases}
  1+\min\{\R(t):s\to t,\ \R(t)\text{ even}\},
     &\text{if an even option exists},\\[2pt]
  1+\max\{\R(t):s\to t\},&\text{otherwise}.
 \end{cases}
\end{equation}
\end{definition}

This is Smith's function \citep{Smith1966}; see also
\citet[vol.~2, Chap.~9]{WinningWays}. Induction gives that $\R$
is even precisely at P-positions. A \emph{remoteness-optimal} move is
one satisfying $\R(t)=\R(s)-1$. Such a move always exists at a
nonterminal position. At N-positions it gives a fastest forced win;
at P-positions it gives the longest resistance.

\begin{definition}\label{def:wyt}
For integers $a,b\geq1$, the \emph{generalized Wythoff game} $\WYT(a,b)$
is the impartial normal-play game whose positions are the pairs
$(x,y)\in\Nzero^2$, and in which a move replaces $(x,y)$ by $(x-e,y-f)$,
where $0\leq e\leq x$, $0\leq f\leq y$, $e+f>0$, and
\begin{equation}\label{eq:legal}
 |e-f|<a\qquad\text{or}\qquad \min(e,f)<b.
\end{equation}
\end{definition}

The case $(a,b)=(1,1)$ is classical Wythoff. Every move
decreases $x+y$, and $(0,0)$ is the only terminal position. The rules
are symmetric in the piles, so $\R(x,y)=\R(y,x)$. We therefore usually
take $x\leq y$, but keep track of both orientations when describing a
move.

For a finite set $S\subset\Nzero$, define
\begin{equation}\label{eq:mex}
 \mexb(S)=\min\{t\in\Nzero:S\cap(t-b,t]=\varnothing\}.
\end{equation}
Equivalently, it is the least nonnegative integer not covered by the
integer intervals $[u,u+b-1]$, $u\in S$. For $b=1$ this is the
usual minimum excludant. If $0\in S$ and its distinct elements are
$s_1<\cdots<s_\ell$, then \eqref{eq:mex} is $s_i+b$ at the first gap
$s_{i+1}-s_i>b$, with $s_{\ell+1}=+\infty$; thus it agrees with
the generalized minimum excludant in \citet{Gurvich2012}.
Directly from \eqref{eq:mex}, $\mexb$ is monotone: if $S\subseteq S'$,
then every $t<\mexb(S)$ is covered by $S$, hence by $S'$, so
$\mexb(S')\geq\mexb(S)$.

Let $(x_m,y_m)$, $m\geq0$, denote the P-positions with $x_m\leq y_m$
and increasing $x_m$. Gurvich's characterization is
\begin{equation}\label{eq:pairs}
 x_m=\mexb\{x_i,y_i:0\leq i<m\},\qquad y_m=x_m+am.
\end{equation}
In particular $(x_0,y_0)=(0,0)$. \citet{Gurvich2012} proved the
spacing bounds, recalled also by \citet{BGMV2024},
\begin{equation}\label{eq:spacing}
 b\leq x_{m+1}-x_m\leq2b,\qquad
 a+b\leq y_{m+1}-y_m\leq a+2b.
\end{equation}
Only the lower bounds are used below, and they follow at once from
\eqref{eq:pairs}: by monotonicity $x_{m+1}\geq x_m$, and $x_{m+1}$ is
not covered by $[x_m,x_m+b-1]$, so $x_{m+1}\geq x_m+b$; then
$y_{m+1}-y_m=x_{m+1}-x_m+a\geq a+b$. Hence both coordinate sequences
are strictly increasing and unbounded.
For an ordered position, writing $d=y-x$, we have
\begin{equation}\label{eq:p-test}
 (x,y)\text{ is a P-position}
 \quad\Longleftrightarrow\quad a\mid d\ \text{ and }\ x=x_{d/a}.
\end{equation}
We use the following remoteness values proved by \citet{BGMV2024}:
\begin{equation}\label{eq:p-remote}
 \R(x_m,y_m)=\R(y_m,x_m)=2m\qquad(m\geq0).
\end{equation}

\section{The remoteness formula}\label{sec:formula}

This section proves the main result. The formula involves a local index
attached to the smaller pile $x$. It records the first P-position, in the
order $(x_0,y_0),(x_1,y_1),\dots$, that has a coordinate in the window of
length $b$ ending at $x$.

\begin{definition}\label{def:index}
The \emph{local index} is the function $\idx$ on $\Nzero$ given by
\begin{equation}\label{eq:index}
 \idx(x)=\min\{m\geq0:x_m\in(x-b,x]\ \text{or}\ y_m\in(x-b,x]\}.
\end{equation}
\end{definition}

The next lemma shows that this set is never empty, so $\idx(x)$ is well
defined. 

\begin{lemma}\label{lem:cover}
For every $k\geq0$:
\begin{enumerate}
\item If $0\leq x<x_k$, some coordinate $x_i$ or $y_i$, with $i<k$,
belongs to $(x-b,x]$.
\item No coordinate of index less than $k$ belongs to $(x_k-b,x_k]$.
\end{enumerate}
Consequently $\idx(x)$ exists for every $x\in\Nzero$, $\idx(x_k)=k$ for
every $k\geq0$, and, for all $x,q\in\Nzero$,
\begin{equation}\label{eq:cover-consequence}
 \idx(x)>q\quad\Longrightarrow\quad x\geq x_{q+1}>x_q.
\end{equation}
\end{lemma}

\begin{proof}
Both assertions follow from $x_k$ being the least uncovered integer in
\eqref{eq:mex}, with $S=\{x_i,y_i:i<k\}$. Given $x$, choose $k$ with
$x<x_k$ and apply the first assertion to prove existence. At $x_k$, the
second assertion excludes smaller indices, while the coordinate $x_k$
itself gives index $k$. Finally, if $x<x_{q+1}$, the first assertion
provides a coordinate of index at most $q$ in $(x-b,x]$. Taking the
contrapositive proves \eqref{eq:cover-consequence}.
\end{proof}

For an N-position $(x,y)$, let
\[
 M(x,y)=\{m:\ (x_m,y_m)\text{ or }(y_m,x_m)
                    \text{ is reachable in one move}\}.
\]
This set is nonempty. Equations \eqref{eq:remote} and
\eqref{eq:p-remote} give
\begin{equation}\label{eq:reduction}
 \R(x,y)=1+2\min M(x,y).
\end{equation}

\begin{lemma}\label{lem:least-index}
If $0\leq x\leq y$ is an N-position, then
\begin{equation}\label{eq:least-index}
 \min M(x,y)=\min\!\left\{\idx(x),
                    \left\lfloor\frac{y-x}{a}\right\rfloor\right\}.
\end{equation}
\end{lemma}

\begin{proof}
Write $d=y-x=aq+r$, where $q\geq0$ and $0\leq r<a$, and set
$j=\idx(x)$.

We first exclude every index $m<\min(j,q)$. Consider a move from $(x,y)$
to the ordered P-position $(x_m,y_m)$, and let $e=x-x_m$ and $f=y-y_m$ be
the numbers of tokens removed from the first and second pile. Then
\[
 f-e=d-am\geq a.
\]
Such a move cannot satisfy $|e-f|<a$. Its legality would therefore
require $e=\min(e,f)<b$. This puts $x_m=x-e$ in $(x-b,x]$, contradicting
$m<j$.
For a move to the reversed position $(y_m,x_m)$ we have $e=x-y_m$ and
$f=y-x_m$, so instead
\[
 f-e=d+am\geq a,
\]
because $m<q$ implies $q\geq1$. Legality again requires $e<b$,
which puts $y_m$ in $(x-b,x]$ and gives the same contradiction.
Thus $\min M(x,y)\geq\min(j,q)$.

We now construct a move to index $\min(j,q)$.

\emph{Case $j\leq q$.} By the definition of $j$, at least one of $x_j$
and $y_j$ lies in $(x-b,x]$. If $y_j\in(x-b,x]$, move to $(y_j,x_j)$.
Both coordinates fit, since $x_j\leq y_j\leq x\leq y$, and the reduction
of the first pile is less than $b$. Otherwise $x_j\in(x-b,x]$, and we
move to $(x_j,y_j)$. Indeed,
\[
 y_j=x_j+aj\leq x+aq\leq y,
\]
and again the first-pile reduction is less than $b$. Either move is
nonzero because the starting position is an N-position.

\emph{Case $j>q$.} Lemma~\ref{lem:cover} gives $x>x_q$. Move to
$(x_q,y_q)$. The reductions are
\[
 e=x-x_q>0,\qquad
 f=y-y_q=e+r\geq0,
\]
so $|e-f|=r<a$. This is a legal move of index $q$.

In both cases $\min(j,q)\in M(x,y)$, proving the reverse inequality.
\end{proof}

\begin{theorem}\label{thm:main}
For every $a,b\geq1$ and $0\leq x\leq y$,
\begin{equation}\label{eq:main}
 \boxed{\displaystyle
 \R(x,y)=2\min\!\left\{\idx(x),
                  \left\lfloor\frac{y-x}{a}\right\rfloor\right\}
                  +\eps(x,y),}
\end{equation}
where $\eps(x,y)=0$ at P-positions and $\eps(x,y)=1$ at N-positions.
By symmetry, $\R(x,y)=\R(y,x)$ gives the value when $x>y$.
\end{theorem}

\begin{proof}
At an N-position use \eqref{eq:reduction} and
Lemma~\ref{lem:least-index}. At a P-position $(x_m,y_m)$,
Lemma~\ref{lem:cover} gives $\idx(x_m)=m$, and
$(y_m-x_m)/a=m$. Thus the right-hand side is $2m$, as required by
\eqref{eq:p-remote}. This includes the terminal position.
\end{proof}

\begin{corollary}[Optimal moves]\label{cor:moves}
At an N-position, the constructions in Lemma~\ref{lem:least-index}
give a remoteness-optimal move. At a nonterminal P-position
$(x_m,y_m)$, reducing the larger pile by one is remoteness-optimal.
\end{corollary}

\begin{proof}
The N-position assertion follows from \eqref{eq:reduction}.
The second assertion also appears in the proof of the P-position
remoteness formula in \citet{BGMV2024}; we recover it here from
Theorem~\ref{thm:main}. We have $m\geq1$, and the legal move leads to
$(x_m,y_m-1)$, still in increasing order since $am\geq1$. This is an
N-position: by \eqref{eq:p-test}, it could only be a P-position if
$a\mid am-1$, that is, $a=1$, and then it would equal
$(x_{m-1},y_{m-1})$, which is impossible because $x_{m-1}<x_m$.
Moreover $\idx(x_m)=m$, and
\[
 \left\lfloor\frac{y_m-1-x_m}{a}\right\rfloor=m-1.
\]
Theorem~\ref{thm:main} gives remoteness $2m-1$ there.
\end{proof}

\begin{remark}[Evaluating the local index]\label{rem:compute}
By the lower bounds in \eqref{eq:spacing}, each coordinate sequence
contributes at most one element to $(x-b,x]$. Let
\[
 L(x)=\max\{m:x_m\leq x\},\qquad
 U(x)=\max\{m:y_m\leq x\}.
\]
Keep $L(x)$ when $x-x_{L(x)}<b$, and keep $U(x)$ when
$x-y_{U(x)}<b$. Then $\idx(x)$ is the smaller retained index, with
at least one retained by Lemma~\ref{lem:cover}. Thus the formula uses
two inverse-sequence queries at the smaller pile. For fixed $a$ and
$b$, the Perron--Frobenius algorithm of \citet{BGO2013} computes
$x_M$ for given $M$, and the largest $m$ with $x_m\leq X$ or with
$y_m\leq Y$ for given $X$ and $Y$, in time polynomial in the binary
input length; these are exactly the queries needed here, and
\citet{BGMV2024} use the same three queries. The new formula
simplifies the selection of a fastest winning move;
polynomial-time computability itself was already established by
\citet{BGMV2024}.
\end{remark}

\section{Explicit formulas when \texorpdfstring{$b=1$}{b=1}}\label{sec:beatty}

For $a\geq1$, define
\[
 \alpha_a=\frac{2-a+\sqrt{a^2+4}}2,\qquad
 \beta_a=\alpha_a+a.
\]
Fraenkel's formulas \citep{Fraenkel1982} are
\begin{equation}\label{eq:beatty}
 x_m=\lfloor\alpha_a m\rfloor,\qquad
 y_m=\lfloor\beta_a m\rfloor.
\end{equation}
The numbers $\alpha_a,\beta_a>1$ are irrational and satisfy
$\alpha_a^{-1}+\beta_a^{-1}=1$. Hence their positive Beatty
sequences partition the positive integers
\citep{Beatty1926,OHA1927}.
Let $\iota_a(0)=0$ and, for $x>0$, let $\iota_a(x)$ be the unique
index of $x$ in these two sequences. It can be evaluated without
searching the sequences: set $k=\lceil x/\alpha_a\rceil$, and use
\begin{equation}\label{eq:beatty-index}
 \iota_a(x)=
 \begin{cases}
 k,&\lfloor\alpha_a k\rfloor=x,\\
 \lceil x/\beta_a\rceil,&\lfloor\alpha_a k\rfloor\ne x.
 \end{cases}
\end{equation}
Indeed, if $x=\lfloor\gamma m\rfloor$ with $\gamma>1$ irrational,
then $m-1<x/\gamma<m$, so $m=\lceil x/\gamma\rceil$.

\begin{corollary}\label{cor:beatty}
In $\WYT(a,1)$,
\[
 \R(x,y)=2\min\!\left\{\iota_a(x),
                 \left\lfloor\frac{y-x}{a}\right\rfloor\right\}
                 +\eps(x,y),\qquad 0\leq x\leq y.
\]
In particular, for classical Wythoff, $\alpha_1=\phi$,
$\beta_1=\phi^2$, and
\begin{equation}\label{eq:classical}
 \R(x,y)=2\min\{\iota_1(x),y-x\}+\eps(x,y).
\end{equation}
In this case $\eps(x,y)=0$ precisely when
$x=\lfloor\phi(y-x)\rfloor$.
\end{corollary}

\begin{proof}
For $b=1$ the window $(x-b,x]$ contains the single integer $x$,
so $\iota_{a,1}(x)=\iota_a(x)$. Apply Theorem~\ref{thm:main}
and \eqref{eq:p-test}.
\end{proof}

The first classical P-positions are
\[
 (0,0),\ (1,2),\ (3,5),\ (4,7),\ (6,10),\ (8,13),\ (9,15).
\]
Table~\ref{tab:examples} illustrates how either term of the minimum
can determine the fastest win. Targets are written in the physical
order of the piles: for instance, the target $(7,4)$ from $(7,12)$
leaves the first pile at $7$ and reduces the second pile to $4$,
reaching the P-position $(4,7)$ with the piles interchanged.

\begin{table}[ht]
\centering
\caption{Classical Wythoff examples and remoteness-optimal moves.
Targets are given in the physical order of the two piles.}
\label{tab:examples}
\begin{tabular}{@{}rrrrl@{}}
\toprule
$(x,y)$ & $\iota_1(x)$ & $y-x$ & $\R(x,y)$ & One optimal target\\
\midrule
$(4,9)$  & 3 & 5 & 7  & $(4,7)$\\
$(7,9)$  & 3 & 2 & 5  & $(3,5)$\\
$(7,12)$ & 3 & 5 & 7  & $(7,4)$\\
$(8,13)$ & 5 & 5 & 10 & $(8,12)$\\
\bottomrule
\end{tabular}
\end{table}

\begin{example}[A window with two indices]
For $\WYT(1,2)$, the first P-positions are
\[
 (0,0),\ (2,3),\ (5,7),\ (9,12),\ (11,15).
\]
The interval $(10,12]$ contains $x_4=11$ and $y_3=12$.
Thus $\iota_{1,2}(12)=3$, and $\R(12,20)=7$, attained by the
move to $(12,9)$.
\end{example}

\section{Independent failed moves}\label{sec:random}

Let $G$ be any finite impartial normal-play game, and fix $p\in(0,1]$.
At a nonterminal position $s$, the player chooses an option $t$. After
each attempted move is chosen, its success is determined by a fresh
Bernoulli($p$) random variable, independent of the preceding history and
of the current choice; thus, given the history and the choice, the attempt
succeeds with probability $p$. If it succeeds, the position becomes $t$;
otherwise it stays at $s$. In both cases the turn passes to the opponent.
At a terminal position the player to move loses immediately; no attempt
is made there.

Write $w_p(s)$ for the optimal probability that the player to move
at $s$ wins. The strategies may depend on the entire observed history.

\begin{lemma}\label{lem:bellman}
The randomized game terminates almost surely under every strategy
pair, has a value, and admits optimal deterministic stationary
strategies. Its value is zero at terminals and satisfies
\begin{equation}\label{eq:bellman}
 w_p(s)=\frac{1-p\min_{s\to t}w_p(t)}{2-p}
\end{equation}
at every nonterminal position. Moreover, an attempt $s\to t$ at a
nonterminal position $s$, followed by optimal play, gives the player to
move the winning probability
$w_p(s)-p\bigl(w_p(t)-\min_{s\to u}w_p(u)\bigr)$; in particular it is
optimal if and only if $w_p(t)=\min_{s\to u}w_p(u)$.
\end{lemma}

\begin{proof}
Let $h(s)$ be the length of the longest path from $s$ in the underlying
acyclic graph. Each success moves the position along an edge, so after
at most $h(s)$ successes the position is terminal and play has ended.
By the model, the successive success indicators are independent
Bernoulli($p$) variables (if play ends, continue the sequence with fresh
ones), and such a sequence contains infinitely many successes almost
surely because $p>0$. Hence play terminates almost surely under every
pair of strategies, which may be randomized and may depend on the whole
history. In particular, the winning probabilities of the two players
always add up to $1$.

Define $w$ on positions recursively on the height: $w(s)=0$ at terminal
positions, and at a nonterminal position
\[
 w(s)=\frac{1-p\,m(s)}{2-p},\qquad m(s)=\min_{s\to t}w(t),
\]
which makes sense because every option of $s$ has smaller height. For
each nonterminal position $u$ fix an option $t^*(u)$ with
$w(t^*(u))=m(u)$, and let $\sigma^*$ be the deterministic stationary
strategy that attempts $t^*(u)$ whenever the position is $u$. We prove by
induction on $h(s)$ the following claim for the game started at $s$,
against every strategy of the opponent, which may be randomized and may
depend on the whole history: if a player uses $\sigma^*$, then he wins
with probability at least $w(s)$ when he is the player to move at $s$, and
with probability at least $1-w(s)$ when he is not.

If $s$ is terminal, the player to move loses at once and $w(s)=0$, so the
claim holds. Let $s$ be nonterminal, assume the claim at all options of
$s$, and write $m=m(s)$ and $t^*=t^*(s)$. Call the player to move at $s$
Player~1 and the other player Player~2. While the position remains $s$,
the attempts alternate, Player~1 making the odd-numbered ones. Let $K$ be
the number of the first successful attempt. Each attempt at $s$ is a fresh
Bernoulli($p$) trial, whatever the history and the choice made, so $K$ is
geometric with parameter $p$ whatever the strategies are:
\[
 \mathbb P(K\text{ odd})=\sum_{k\ \mathrm{odd}}p(1-p)^{k-1}=\frac1{2-p},
 \qquad
 \mathbb P(K\text{ even})=\frac{1-p}{2-p}.
\]
After the first success the position is an option of $s$. Because
$\sigma^*$ is stationary, the continuation of a player who uses it is
$\sigma^*$ in the game started at the new position, and the opponent's
conduct, conditioned on the history so far, is some strategy of that game.
This is why the claim at the options of $s$ can be applied.

\emph{Player~1 uses $\sigma^*$.} If $K$ is odd, Player~1 has moved to
$t^*$ and Player~2 is to move there, so by the claim at $t^*$ Player~1
wins with probability at least $1-w(t^*)=1-m$. If $K$ is even, Player~2
has moved to some option $t$ and Player~1 is to move there, so Player~1
wins with probability at least $w(t)\geq m$. Hence he wins with
probability at least
\[
 \frac{1-m}{2-p}+\frac{(1-p)\,m}{2-p}=\frac{1-pm}{2-p}=w(s).
\]

\emph{Player~2 uses $\sigma^*$.} If $K$ is odd, Player~1 has moved to some
option $t$ and Player~2 is to move there, so Player~2 wins with
probability at least $w(t)\geq m$, and Player~1 wins with probability at
most $1-m$. If $K$ is even, Player~2 has moved to $t^*$ and Player~1 is to
move there, so by the claim at $t^*$ Player~2 wins with probability at
least $1-w(t^*)=1-m$, and Player~1 wins with probability at most $m$.
Hence Player~1 wins with probability at most
\[
 \frac{1-m}{2-p}+\frac{(1-p)\,m}{2-p}=w(s),
\]
that is, Player~2 wins with probability at least $1-w(s)$. This completes
the induction.

Consequently the game started at $s$ has a value, which equals $w(s)$,
and the strategy $\sigma^*$ is optimal for both players. The value of a
game is unique, so $w=w_p$, and $w_p$ satisfies \eqref{eq:bellman} by
construction. The value is zero at terminal positions, and the optimal
strategies $\sigma^*$ are deterministic and stationary.

Finally, suppose the player to move at $s$ attempts an option $t$ and
plays optimally afterwards. With probability $p$ the attempt succeeds and
the opponent is to move at $t$, and with probability $1-p$ it fails and
the opponent is to move at $s$ again. So the attempt gives the winning
probability
\[
 p\bigl(1-w_p(t)\bigr)+(1-p)\bigl(1-w_p(s)\bigr)
 =w_p(s)-p\bigl(w_p(t)-m\bigr),
\]
by \eqref{eq:bellman}. Since $p>0$ and $w_p(t)\geq m$, this equals
$w_p(s)$ if and only if $w_p(t)=m$. Thus the optimal attempts at $s$ are
exactly the attempts at options that minimize $w_p$.
\end{proof}

\begin{theorem}\label{thm:random}
Let $\lambda=p/(2-p)$. At every position of $G$,
\begin{equation}\label{eq:random}
 \boxed{\displaystyle
 w_p(s)=\frac{1-(-\lambda)^{\R(s)}}2
       =\frac12+\frac{(-1)^{\R(s)+1}}2\lambda^{\R(s)}.}
\end{equation}
For $0<p<1$, an attempted move is optimal if and only if it decreases
$\R$ by one. Hence the sets of optimal attempts are independent of
$p$ throughout this interval.
\end{theorem}

\begin{proof}
Put $u_p=w_p-1/2$. Equation~\eqref{eq:bellman} becomes
\begin{equation}\label{eq:centered}
 u_p(s)=-\lambda\min_{s\to t}u_p(t),\qquad
 u_p(s)=-\tfrac12\text{ at terminals}.
\end{equation}
Suppose first $0<p<1$, so $0<\lambda<1$. Induct on the height of
$s$. The terminal case follows from $\R(s)=0$. If some option has
even remoteness, the negative successor values are exactly those of
even remoteness. Among them the least is
$-\lambda^r/2$, where $r$ is the smallest even successor remoteness.
Thus \eqref{eq:centered} gives $u_p(s)=\lambda^{r+1}/2$, and
$\R(s)=r+1$ by \eqref{eq:remote}.

If all options have odd remoteness, all their centered values are
positive. The least is $\lambda^r/2$, where $r$ is the largest
successor remoteness. Equation~\eqref{eq:centered} gives
$u_p(s)=-\lambda^{r+1}/2$, and again $\R(s)=r+1$.
This proves the formula. The strict monotonicity of the powers of
$\lambda$ shows in both cases that the minimizing options are
exactly those of remoteness $\R(s)-1$. Lemma~\ref{lem:bellman}
then proves the assertion about optimal attempts.

For $p=1$, the game is deterministic. Formula~\eqref{eq:random}
reduces to $w_1(s)=0$ at even remoteness and $w_1(s)=1$ at odd
remoteness, which is the P/N classification.
\end{proof}

\Needspace{9\baselineskip}
\begin{corollary}\label{cor:invariance}
For every $p\in(0,1]$, the value is less than $1/2$ precisely at
the deterministic P-positions and greater than $1/2$ precisely at
the deterministic N-positions. In particular no position has value
exactly $1/2$. For fixed $0<p<1$, remoteness can be recovered from
the value by
\[
 \R(s)=\frac{\log|2w_p(s)-1|}{\log(p/(2-p))}.
\]
\end{corollary}

\begin{proof}
Use \eqref{eq:random} and $\lambda>0$; for the second assertion take
logarithms of $|2w_p(s)-1|=\lambda^{\R(s)}$.
\end{proof}

\begin{proposition}[Duration under optimal play]\label{prop:duration}
Suppose $0<p<1$ and both players use optimal stationary strategies.
Starting from $s$ with $r=\R(s)>0$, the number $T$ of attempted moves
has distribution
\begin{equation}\label{eq:duration}
 \mathbb P(T=n)=\binom{n-1}{r-1}p^r(1-p)^{n-r},\qquad n\geq r.
\end{equation}
In particular,
\[
 \mathbb E[T]=\frac rp,\qquad
 \Var(T)=\frac{r(1-p)}{p^2}.
\]
At a terminal position, $T=0$.
\end{proposition}

\begin{proof}
Every successful optimal attempt decreases remoteness by one, while
a failed attempt leaves it unchanged. The game therefore ends on
the $r$th success. This is the negative binomial waiting-time
distribution \eqref{eq:duration}, or equivalently a sum of $r$
independent geometric waiting times, each with mean $1/p$ and
variance $(1-p)/p^2$.
\end{proof}

This gives another interpretation of \eqref{eq:random}. The player
initially to move wins exactly when $T$ is odd, and
\[
 \mathbb E[(-1)^T]=\left(-\frac{p}{2-p}\right)^r.
\]
Thus the probability of winning is
$(1-\mathbb E[(-1)^T])/2$. Theorem~\ref{thm:random} establishes
the optimality needed for this interpretation.

For $p=1$, remoteness-optimal play is still optimal, but other
value-optimal moves can have different durations. Also, $p=0$ is
excluded: starting from a nonterminal position, play would never
terminate. Although $w_p(s)\to1/2$ as $p\downarrow0$ for each
nonterminal $s$, this limit does not define a normal-play value at
$p=0$.

Combining Theorems~\ref{thm:main} and \ref{thm:random}, let
\[
 k=\min\!\left\{\idx(x),\left\lfloor\frac{y-x}{a}\right\rfloor\right\}.
\]
Then randomized $\WYT(a,b)$ has the explicit expression in P-position
coordinates
\begin{equation}\label{eq:random-wyt}
 w_p(x,y)=\frac{1-(-\lambda)^{2k+\eps(x,y)}}2.
\end{equation}
The moves in Corollary~\ref{cor:moves} are optimal for every
$p\in(0,1]$. For example, $\R(4,9)=7$ in classical Wythoff. At
$p=1/2$ its value is $1094/2187$, and its mean duration under optimal
play is $14$ attempted moves.

\section{Reproducible checks}\label{sec:checks}

The proofs above do not depend on computation. The supplementary material contains the Python program \texttt{verify\_remoteness.py}, which
checks the formulas and the stated move constructions independently.
It generates P-positions from \eqref{eq:pairs}, but computes remoteness
separately by enumerating legal options and applying \eqref{eq:remote}.
It also checks that the even values of this independent recursion
agree with \eqref{eq:p-test}.

The parameter grid is $1\leq a,b\leq8$ with
$0\leq x\leq y\leq70$. Fourteen parameter pairs are additionally
checked up to $y=170$:
\[
\begin{gathered}
 (1,1),(2,1),(1,2),(2,2),(1,3),(3,2),(2,3),\\
 (1,4),(4,4),(3,3),(5,2),(2,5),(4,1),(1,5).
\end{gathered}
\]
Every nonterminal position is also checked for legality of the
constructed move and a decrease of remoteness by exactly one.
There are $333684$ distinct parameter-position cases across these
runs, with no discrepancies in the formula or the move construction.

For the stochastic identity, the program enumerates all $2^{15}=32768$
directed graphs on vertices $0,\ldots,5$ whose edges point to smaller
vertices. This includes, up to a compatible labeling, every acyclic
game graph on at most six vertices. It evaluates the Bellman
recurrence using exact rational arithmetic at
$p\in\{1/5,1/2,4/5,1\}$ and compares with
Theorem~\ref{thm:random}. For $p<1$, it also compares the complete
sets of optimal attempts with the remoteness-decreasing options.
All $786432$ state-probability evaluations agree with the formula,
and all optimal-option comparisons agree.

\section{Concluding remarks}\label{sec:conclusion}

The two-index formula identifies a fastest winning move in generalized Wythoff games directly from the P-position coordinates and the pile difference. The failed-move model gives remoteness a probabilistic interpretation that applies to every finite impartial normal-play game. For a common success probability strictly between zero and one, remoteness determines the optimal attempts, the winning probability, and the duration distribution under optimal stationary play.

\section*{Statements and Declarations}

\paragraph{Conflict of interest}
The authors have no conflicts to disclose.

\paragraph{Data and code availability.}
The verification program \texttt{verify\_remoteness.py} is provided as supplementary material.

\paragraph{Use of generative AI.}
ChatGPT was used as an exploratory and editorial aid, including for
checking the mathematical arguments and for writing the supplementary
verification code. The authors take full responsibility for the
content of the article and the code.

\end{document}